\documentclass{article}
\usepackage{amsmath,amssymb,titling,amsthm,enumitem,mathrsfs,euscript,upgreek,wasysym}
\usepackage[utf8]{inputenc}
\usepackage[margin=1in]{geometry}

\theoremstyle{plain}
\newtheorem{thm}{Theorem}[section]

\makeatletter
\newcommand\varitem[1]{\item[\textbf{A\arabic{enumi}\rlap{$#1$}.}]%
  \edef\@currentlabel{A\arabic{enumi}{$#1$}}}
\makeatother

\theoremstyle{definition}
\newtheorem{defn}{Definition}[section]

\theoremstyle{remark}

\newtheorem*{note}{Note}

\begin{document}
{\centering\scshape\Large\textsc{Transfinite Galois Theory} \par}
{\centering\scshape\textsc{Alec Rhea} \par}

\vspace{5mm}

\begin{abstract}
In this paper I generalize the notion of a polynomial over an ordered field to that of a naked polynomial over a non-Archimedean ordered field, subsequently showing that the notion of a naked polynomial ring forms an Euclidean domain. This canonically generalizes the methods of Galois theory of fields and polynomial rings to a transfinite Galois theory of non-Archimedean ordered fields and naked polynomial rings, lifting the processes of splitting and algebraic closure to non-Archimedean ordered fields. \\
\end{abstract}

\section{Transfinite Galois Theory}
A Galois theory of the Surrational numbers $\mathbb{Q}_\infty$ ([$1$]) is motivated by the fact that there are convex multiplicative subgroups of the Surreals generated by roots of non-finite elements -- as a result, although $\mathbb{Q}_\infty$ is dense in the pieces of the Surreal line where it is defined, there are entire convex segmemts where it is not defined around $\sqrt[\alpha]{\omega}$ for all $\alpha>2$ (it is defined to the right of $\sqrt{\omega})$ and other infinite/infinitesimal roots. Accordingly, we now begin generalizing classical Galois theory to a fundamentally transfinite version necessary for constructing field extensions which contain these infinite and infinitesimal roots over $\mathbb{Q}_\infty$. The first step in this process is defining an appropriately generalized notion of a polynomial and polynomial ring, which can take on positive Surinteger ([$1$]) values in its 'exponents' instead of just natural numbers. \\

\begin{defn}
Let $\mathbb{F}$ be an ordered field, and let ${\mathbb{G}^+}$ be the nonnegative part of a discretely ordered group. We define $\mathbb{F}[X^{\mathbb{G}^+}]$ to be $\mathbb{F}^{\mathbb{G}^+}_{\subseteq max}$, the ordered group algebra whose elements are functions $f:{\mathbb{G}^+}\rightarrow\mathbb{F}$ such that all subsets $\mathbb{A}\subseteq supp(f)\subseteq\mathbb{G}^+$ have a maximal element, together with the standard algebraic structure and a total ordering structure. More precisely, we define $$\mathbb{F}[X^{\mathbb{G}^+}]=\{f\in\mathbb{F}^{\mathbb{G}^+}:\mathbb{A}\subset supp(f)\implies\exists x\in\mathbb{A}\forall y\in\mathbb{A}\big[y\leq x\big]\}.$$ For $p\in\mathbb{F}[X^{\mathbb{G}^+}]$ we will write $p=(p_\alpha)_{\alpha\in{\mathbb{G}^+}}=(p_0,\dots)$, where $p_\alpha$ is the image of $\alpha\in{\mathbb{G}^+}$ under $p$; we denote by $supp(p)\subseteq{\mathbb{G}^+}$ the set of elements whose image is nonzero. Let $q\in\mathbb{F}[X^{\mathbb{G}^+}]$ as well with $q=(q_\alpha)_{\alpha\in{\mathbb{G}^+}}$. The addition $\hat+:\mathbb{F}[X^{\mathbb{G}^+}]\times\mathbb{F}[X^{\mathbb{G}^+}]\rightarrow\mathbb{F}[X^{\mathbb{G}^+}]$ is given by $$p\hat+q=(p_\alpha+q_\alpha)_{\alpha\in{\mathbb{G}^+}},$$ $$\hat+=\langle p\hat+q:p,q\in\mathbb{F}[X^{\mathbb{G}^+}]\rangle.$$ Negation $\hat -:\mathbb{F}[X^{\mathbb{G}^+}]\rightarrow\mathbb{F}[X^{\mathbb{G}^+}]$ is then given by $$-p=(-p_\alpha)_{\alpha\in{\mathbb{G}^+}},$$ $$\hat-=\langle-p:p\in\mathbb{F}[X^{\mathbb{G}^+}]\rangle.$$ Multiplication $\hat\times:\mathbb{F}[X^{\mathbb{G}^+}]\times\mathbb{F}[X^{\mathbb{G}^+}]\rightarrow\mathbb{F}[X^{\mathbb{G}^+}]$ is then defined coordinate-wise by $$(p\hat\times q)_\gamma=\sum_{\alpha+\beta=\gamma}p_\alpha q_\beta$$ where we have renamed the indices of $q$, yielding $$p\hat\times q=\big((p\hat\times q)_\gamma\big)_{\gamma\in{\mathbb{G}^+},}$$ $$\hat\times=\langle p\hat\times q:p,q\in\mathbb{F}[X^{\mathbb{G}^+}]\rangle.$$ Finally, the ordering on $\mathbb{F}[X^{\mathbb{G}^+}]$, written as $\preceq$, is defined as follows. For all $p,q\in\mathbb{F}[X^{\mathbb{G}^+}]$, let $$_p\uparrow_q=\max\{g\in\mathbb{G}^+:p(g)\neq q(g)\},$$ so $_p\uparrow_q\in\mathbb{G}^+$ is the last coordinate at which $p$ and $q$ differ -- this is well defined since \\ $\{g\in\mathbb{G}^+:p(g)\neq q(g)\}\subseteq\big[supp(p)\cup supp(q)\big]$. We then define $$\preceq=\{(p,q):p(_p\uparrow_q)<q(_p\uparrow_q)\vee p=q\},$$ which is essentially a reverse lexicographic ordering on the functions, where the discretely ordered monoid elements serve as the letter positions. The members of $\mathbb{F}[X^{\mathbb{G}^+}]$ will be called {\bf naked polynomials}, and we will refer to $\mathbb{F}[X^{\mathbb{G}^+}]$ as a {\bf naked polynomial ring} over $\mathbb{F}$.
\end{defn}

\begin{note}
If we view a naked polynomial $p=(p_\alpha)_{\alpha\in{\mathbb{G}^+}}$ as a 'dressed up' polynomial by writing it as $$p=\sum_{\alpha\in{\mathbb{G}^+}}p_\alpha X^\alpha,$$ the above definitions match exactly our intuition for how polynomials should behave. The definitions and notation employed here are to emphasize that the structure underlying the 'exponents' of a polynomial ring is not most correctly described as that of a vector space basis, but rather an indexed collection of discretely ordered positions such that all subsets of filled positions have a maximal element.  Polynomial addition combines elements in identical positions, and multiplication fuses elements from different positions in the classically understood "add the exponents" fashion. Further, $_{\hat0}\uparrow_q$ will play the role of $deg(q)$ for all naked polynomials $q\in\mathbb{F}[X^{\mathbb{G}^+}]$, where $\hat0\in\mathbb{F}[X^{\mathbb{G}^+}]$ is the function with empty support.\\
\end{note}

The notion of a naked polynomial ring is a canonical generalization of a polynomial ring in the sense that a naked polynomial ring is a polynomial ring when the discretely ordered monoid in question is the countable infinity $\omega$ under natural addition, or equivalently the natural numbers. Using $\mathbb{Z}_\lambda^+$ (or all of $\mathbb{Z}_\infty^+$) produces larger non-Archimedean discretely ordered monoids with the necessary property, and consequently a generalized notion of a polynimial ring. \\

\begin{thm}
For any ordered field $\mathbb{F}$ with $\mathbb{Z}_\infty$ viewed as a discretely ordered group under addition, $\mathbb{F}[X^{\mathbb{Z}_\infty^+}]$ is an ordered ring under $\hat+$, $\hat-$, $\hat\times$ and $\preceq$, with additive identity $\hat0=(0)_{\alpha\in \mathbb{Z}_\infty^+}=(0,0,0,\dots)$ and multiplicative identity $\hat1=(1,0,0,\dots)$.
\end{thm}
\begin{proof}
We will first show that $\mathbb{F}[X^{ \mathbb{Z}_\infty^+}]$ is a group under $\hat+$ and $\hat-$; suppose $p,q,r\in\mathbb{F}[X^{ \mathbb{Z}_\infty^+}]$. We then have that $$p\hat+\hat0=(p_\alpha+0)_{\alpha\in { \mathbb{Z}_\infty^+}}=(p_\alpha)_{\alpha\in { O_n}}=p,$$ so $\hat0$ is an additive identity. Further, we have that $$p\hat+q=(p_\alpha+q_\alpha)_{\alpha\in { \mathbb{Z}_\infty^+}}=(q_\alpha+p_\alpha)_{\alpha\in { \mathbb{Z}_\infty^+}}=q\hat+p$$ and $$p\hat+(q\hat+r)=p\hat+(q_\alpha+r_\alpha)_{\alpha\in { \mathbb{Z}_\infty^+}}=\big(p_\alpha+(q_\alpha+r_\alpha)\big)_{\alpha\in { \mathbb{Z}_\infty^+}}$$ $$=\big((p_\alpha+q_\alpha)+r_\alpha\big)=(p_\alpha+q_\alpha)_{\alpha\in { \mathbb{Z}_\infty^+}}\hat+r=(p\hat+q)\hat+r,$$ so $\hat+$ is commutative and associative in $\mathbb{F}[X^{ \mathbb{Z}_\infty^+}]$. Finally, we observe that $$p\hat+(\hat-p)=(p_\alpha-p_\alpha)_{\alpha\in { \mathbb{Z}_\infty^+}}=(0)_{\alpha\in { \mathbb{Z}_\infty^+}}=(0,0,0,\dots)=\hat0,$$ so all elements have unique additive inverses and $\hat-$ is the additive inversion mapping. This completes the proof that $\mathbb{F}[X^{ \mathbb{Z}_\infty^+}]$ is a group under $\hat+$ and $\hat-$. \\ We now show that $\mathbb{F}[X^{ \mathbb{Z}_\infty^+}]$ is a monoid under $\hat\times$. We first observe that $$p\hat\times\hat1=(p_\alpha\times 1)_{\alpha\in { \mathbb{Z}_\infty^+}}=(p_\alpha)_{\alpha\in { \mathbb{Z}_\infty^+}},$$ so $\hat1$ is a multiplicative identity in $\mathbb{F}[X^{ \mathbb{Z}_\infty^+}]$. We now proceed to argue by coordinates; for all $\gamma\in { \mathbb{Z}_\infty^+}$ we have that $$(p\hat\times q)_\gamma=\sum_{\alpha+\beta=\gamma}p_\alpha q_\beta=\sum_{\beta+\alpha=\gamma}q_\beta p_\alpha=(q\hat\times p)_\gamma,$$ thus $$p\hat\times q=\big((p\hat\times q)_\gamma\big)_{\gamma\in { \mathbb{Z}_\infty^+}}=\big((q\hat\times p)_\gamma\big)_{\gamma\in { \mathbb{Z}_\infty^+}}=q\hat\times p,$$ so $\hat\times$ is commutative on $\mathbb{F}[X^{ \mathbb{Z}_\infty^+}]$. Further, we have that $$\big(p\hat\times(q\hat\times r)\big)_\gamma=\sum_{\alpha+\beta=\gamma}p_\alpha(q\hat\times r)_\beta=\sum_{\alpha+\beta=\gamma}p_\alpha\Big(\sum_{\zeta+\nu=\beta}q_\zeta r_\nu\Big)$$ $$=\sum_{\alpha+\beta=\gamma}\big(\sum_{\zeta+\nu=\alpha}p_\zeta q_\nu\big)r_\beta=\sum_{\alpha+\beta=\gamma}(p\hat\times q)_\alpha r_\beta=\big((p\hat\times q)\hat\times r\big)_\gamma,$$ thus $$p\hat\times(q\hat\times r)=\Big(\big(p\hat\times(q\hat\times r)\big)_\gamma\Big)_{\gamma\in { \mathbb{Z}_\infty^+}}=\Big(\big((p\hat\times q)\hat\times r)\big)_\gamma\Big)_{\gamma\in { \mathbb{Z}_\infty^+}}=(p\hat\times q)\hat\times r,$$ so we see that $\hat\times$ is associative on $\mathbb{F}[X^{ \mathbb{Z}_\infty^+}]$, completing the proof that it is a monoid under $\hat\times$. \\ We now show that $\mathbb{F}[X^{ \mathbb{Z}_\infty^+}]$ is a ring under $\hat+$, $\hat-$ and $\hat\times$ by showing that multiplication distributes well over addition. Indeed, we have that $$p\hat\times(q\hat+ r)=p\hat\times(q_\alpha+r_\alpha)_{\alpha\in { \mathbb{Z}_\infty^+}},$$ and for all $\gamma\in  \mathbb{Z}_\infty^+$ we then have that $$\big(p\hat\times(q_\alpha+q_\alpha)_{\alpha\in { \mathbb{Z}_\infty^+}}\big)_\gamma=\sum_{\beta+\nu=\gamma}p_\beta(q_\nu+r_\nu)=\sum_{\beta+\nu=\gamma}p_\beta q_\nu+p_\beta r_\nu$$ $$=\sum_{\beta+\nu=\gamma}p_\beta q_\nu+\sum_{\beta+\nu=\gamma}p_\beta r_\nu=(p\hat\times q)_\gamma+(p\hat\times r)_\gamma.$$ Consequently, $$p\hat\times(q\hat+r)=\Big(\big(p\hat\times(q_\alpha+r_\alpha)_{\alpha\in { \mathbb{Z}_\infty^+}})_\gamma\Big)_{\gamma\in { \mathbb{Z}_\infty^+}}=\big((p\hat\times q)_\gamma+(p\hat\times r)_\gamma\big)_{\gamma\in { \mathbb{Z}_\infty^+}}$$ $$=\big((p\hat\times q)_\gamma\big)_{\gamma\in { \mathbb{Z}_\infty^+}}\hat+\big((p\hat\times r)_\gamma\big)_{\gamma\in { \mathbb{Z}_\infty^+}}=(p\hat\times q)\hat+(p\hat\times r),$$ completing the proof that $\mathbb{F}[X^{ \mathbb{Z}_\infty^+}]$ is a ring under $\hat+$, $\hat-$ and $\hat\times$. \\ We now show that $\mathbb{F}[X^{ \mathbb{Z}_\infty^+}]$ is an ordered ring; suppose that $p\prec q$, so $p_{_p\uparrow_q}<q_{_p\uparrow_q}$. We then observe that $_{p+r}\uparrow_{q+r}=_p\uparrow_q$, since the $_p\uparrow_q^{th}$ coordinate of $p$ already differs from the $_p\uparrow_q^{th}$ coordinate of $q$, so adding the $_p\uparrow_q^{th}$ coordinate of $r$ can't make them match and all higher coordinates will match by definition. Consequently, we have that $$(p\hat+r)_{_{p+r}\uparrow_{q+r}}=(p\hat+r)_{_p\uparrow_q}=p_{_p\uparrow_q}+r_{_p\uparrow_q}<q_{_p\uparrow_q}+r_{_p\uparrow_q}=(q+r)_{_p\uparrow_q}=(q+r)_{_{p+r}\uparrow_{q+r}},$$ thus $p\hat+r\prec q\hat+r$ as desired. Now suppose that $\hat0\prec p$, so $\hat0\prec q$ as well. We then have that the last nonzero coordinate of $p$ and $q$ must be positive, and the last nonzero coordinate of $p\hat\times q$ will be the product of these two coordinates, so it will also be positive, thus $\hat0\prec p\hat\times q$.  This completes the proof that $\mathbb{F}[X^{ \mathbb{Z}_\infty^+}]$ is an ordered ring under $\hat+$, $\hat-$, $\hat\times$ and $\preceq$. \\
\end{proof}

Note that $\mathbb{F}[X^{ \mathbb{Z}_\infty^+}]$ is a proper class; accordingly, we will refer to $\mathbb{F}[X^{ \mathbb{Z}_\infty^+}]$ as a naked polynomial Ring. Further, for any $\gamma$-number $\omega^\eta$ with $\eta\in O_n$ fixed we have that $\mathbb{F}[X^{\mathbb{Z}_{\omega^\eta}^+}]$, where we view $\mathbb{Z}_{\omega^\eta}^+$ as an ordered additive monoid, is also a proper subring of $\mathbb{F}[X^{\mathbb{Z}_\infty^+}]$. \\

\vspace{5mm}

\begin{thm}
Let $\eta\in O_n$ be fixed; then $\mathbb{F}[X^{\mathbb{Z}_{\omega^\eta}^+}]$ is a proper subring of $\mathbb{F}[X^{\mathbb{Z}_\infty^+}]$.
\end{thm}
\begin{proof}
That $\mathbb{F}[X^{\omega^\eta}]$ has the structure of an ordered ring except closure follows immedately from Theorem {\it 4.1}, with closure following from the fact that $\omega^\eta$ is a $\gamma$-number for all $\eta\in\mathbb{O}_n$, so $\alpha,\beta<\omega^\eta\implies\alpha+\beta<\omega^\eta$. This completes the proof. \\
\end{proof}

\begin{note}
$$\mathbb{F}[X^{\mathbb{Z}_\infty^+}]=\bigcup_{\alpha\in O_n}\mathbb{F}[X^{\mathbb{Z}_{\omega^\alpha}^+}].$$ \\
\end{note}

We now define a generalization of the Kroneker-delta index symbol $\delta_{ij}$ which will allow us to translate any 'dressed up' polynomial in $\mathbb{Q}_\infty[X^{\mathbb{Z}_\infty^+}]$ into an equivalent naked representation. \\

\begin{defn}
We define a function $[\ ]:\mathbb{Z}_\infty^+\times\mathbb{Q}_\infty\times\mathbb{Z}_\infty^+\rightarrow\mathbb{Q}_\infty$ by $$ _m[q]_n=\begin{cases} q,&{\text if}\ m=n, \\ 0,&{\text if}\ m\neq n.\end{cases}$$ \\
\end{defn}

\begin{note}
We now have that for all $\mathfrak{p}\in\mathbb{Q}_\infty[X^{\mathbb{Z}_\infty^+}]$, we may dress or undress $\mathfrak{p}$ as we see fit using $$\mathfrak{p}=\sum_{n<\gamma}c_n X^{g_n}=\Big(\sum_{n<\gamma}c_n\delta_{\beta g_n}\Big)_{\beta\in\mathbb{Z}_\infty^+}=\Big(\sum_{n<\gamma} {_\beta}[c_n]_{g_n}\Big)_{\beta\in\mathbb{Z}_\infty^+}=\hat\sum_{n<\gamma}({_\beta[c_n]_{g_n}})_{\beta\in\mathbb{Z}_\infty^+},$$ where $\sum_{n<\gamma} {_\beta}[c_n]_{g_n}$ and $({_\beta[c_n]_{g_n}})_{\beta\in\mathbb{Z}^+}$ each have at most one nonzero term, when $g_n=\beta$ ($g$ is injective in general). Additionally, $_i[1]_j=\delta_{ij}$ and $q {_i[1]_j}={_i[q]_j}$ for all $q\in\mathbb{Q}_\infty$ and $i,j\in\mathbb{Z}_\infty^+$. Further, for $\mathfrak{p}=\Big(\sum_{n<\gamma}{_\beta[c_n]_{g_n}}\Big)_{\beta\in\mathbb{Z}_\infty}$ and $\mathfrak{q}=\Big(\sum_{m<\zeta}{_\beta[d_m]_{f_m}}\Big)_{\beta\in\mathbb{Z}_\infty}$, we now have a canonical representation for $\mathfrak{p}\hat\times\mathfrak{q}$: $$\mathfrak{p}\hat\times\mathfrak{q}=\Big(\sum_{n<\gamma}{_\beta[c_n]_{g_n}}\Big)_{\beta\in\mathbb{Z}_\infty}\hat\times\Big(\sum_{m<\zeta}{_\beta[d_m]_{f_m}}\Big)_{\beta\in\mathbb{Z}_\infty}=\Big(\sum_{\ell<\gamma+\zeta}\big(\sum_{g_n+f_m=\ell}{_\beta[c_nd_m]_\ell}\big)\Big)_{\beta\in\mathbb{Z}_\infty^+}.$$ The intuition for this can be checked surprisingly easily using dressed polynomial multiplication. \\
\end{note}

\begin{defn}
Let $\mathbb{F}[X^{\mathbb{G}^+}]$ be a naked polynomial ring. From now on, lowercase Fraktur script lettering from the middle of the alphabet $\mathfrak{p},\mathfrak{q},\mathfrak{r},\dots$ will denote a naked polynomial in $\mathbb{F}[X^{\mathbb{G}^+}]$. For a naked polynomial $\mathfrak{p}\in\mathbb{F}[X^{\mathbb{G}^+}]$, we define the class of {\bf factors} of $\mathfrak{p}$, $\mathcal{F}_\mathfrak{p}$, as $$\mathcal{F}_\mathfrak{p}=\{\{\mathfrak{q},\mathfrak{r}\}:\mathfrak{q}\hat\times\mathfrak{r}=\mathfrak{p}\},$$ and if $\{\mathfrak{q},\mathfrak{r}\}\in\mathcal{F}_\mathfrak{p}$ then we will say that $\mathfrak{q}$ and $\mathfrak{r}$ {\bf split} $\mathfrak{p}$, and refer to them as {\bf factors} of $\mathfrak{p}$. We will say that a naked polynomial $\mathfrak{q}$ is a {\bf root} iff $\mathfrak{q}=(q,1,0,\dots)$ for some fixed $q\in\mathbb{F}$, and we will say that a naked polynomial {\bf has a root} iff at least one of its factors is a root. Additionally, we will say that a naked polynomial $\mathfrak{p}$ {\bf splits} iff all of its factors are roots or $\hat1$. We will also say that a naked polynomial $\mathfrak{p}\in\mathbb{F}[X^{\mathbb{G}^+}]$ is {\bf irreducible} iff $\mathcal{F}_\mathfrak{p}=\{\{\mathfrak{p},\hat1\}\}$. Finally, we define the {\bf ideal generated by} $\mathfrak{p}\in\mathbb{F}[X^{\mathbb{G}^+}]$, denoted $\mathcal{I}_\mathfrak{p}$, as $$\mathcal{I}_\mathfrak{p}=\{\mathfrak{q}:\exists\{x,y\}[\mathfrak{p}\in\{x,y\}\in\mathcal{F}_\mathfrak{q}]\}.$$  \\
\end{defn}

In general, to real-algebraically close $\mathbb{Q}_\lambda$ we should split all irreducible naked polynomials in $\mathbb{Q}_\lambda[X^\lambda]$ except $$\mathfrak{i}=\Big(\sum_{n<2}{_\alpha[1]_{2n}}\Big)_{\alpha\in\mathbb{Z}_\infty^+}=({_\beta[1]_0}+{_\beta[1]_2})_{\alpha\in\mathbb{Z}_\infty^+}=(1,0,1,0,0,0,\dots),$$ the naked equivalent of the infamous complex-closing $x^2+1$. Note that if we heuristically define $i$ by $i^2=-1$ then $$\mathfrak{i}=(1,0,1,0,0,0,\dots)=(i,1,0,0,\dots)\hat\times(-i,1,0,0,\dots),$$ which is the naked equivalent of $x^2+1=(x+i)(x-i)$. To real-algebraically close all of $\mathbb{Q}_\infty$, we need to use all of the irreducible polynomials in $\mathbb{Q}_\infty[X^{{\mathbb{Z}_\infty^+}}]$ except $\mathfrak{i}$. In general, we will introduce the $\alpha^{th}$ root of a surrational number $p\in\mathbb{Q}_\infty$ that doesn't already have one by splitting the irreducible naked polynomial $\mathfrak{p}_\alpha=({_\beta[-p]_0}+{_\beta[1]_\alpha})_{\beta\in\mathbb{Z}_\infty^+}=(-p,\dots,1,\dots)\in\mathbb{Q}_\infty[X^{{\mathbb{Z}_\infty^+}}]$, where $1$ is placed in the $\alpha^{th}$ coordinate position; this is the naked equivalent of $X^\alpha-p$. This allows us to move into the ordered field extension $\mathbb{Q}_\infty[X^{{\mathbb{Z}_\infty^+}}]\setminus\mathcal{I}_{\mathfrak{p}_\alpha}$ over $\mathbb{Q}_\infty$. \\

\begin{defn}
Let $\mathcal{I}\subseteq\mathbb{Q}_\infty[X^{{\mathbb{Z}_\infty^+}}]$ be an ideal. We define a congruence relation $[\mathcal{I}]\subseteq\mathbb{Q}_\infty[X^{{\mathbb{Z}_\infty^+}}]\times\mathbb{Q}_\infty[X^{{\mathbb{Z}_\infty^+}}]$ by $$[\mathcal{I}]=\{(\mathfrak{p},\mathfrak{q}):\mathfrak{p}-\mathfrak{q}\in\mathcal{I}\}.$$ $[\mathcal{I}]$ will be called the {\bf congruence relation induced by} $\mathcal{I}$, and we will write $\mathfrak{p}[\mathcal{I}]\mathfrak{q}$ iff $(p,q)\in[\mathcal{I}]$. Further, since congruence relations are equivalence relations, we denote by $[\mathcal{I}]_\mathfrak{p}$ the equivalence class generated by $[\mathcal{I}]$ and $\mathfrak{p}\in\mathbb{Q}_\infty[X^{{\mathbb{Z}_\infty^+}}]$; that is,  for all $\mathfrak{p}\in\mathbb{Q}_\infty[X^{{\mathbb{Z}_\infty^+}}]$ we define $$[\mathcal{I}]_\mathfrak{p}=\{\mathfrak{q}:\mathfrak{q}[\mathcal{I}]\mathfrak{p}\}.$$ We then define the {\bf ideal quotient of $\mathbb{Q}_\infty[X^{{\mathbb{Z}_\infty^+}}]$ by $\mathcal{I}$}, denoted $\mathbb{Q}_\infty[X^{{\mathbb{Z}_\infty^+}}]\setminus\mathcal{I}$, by $$\mathbb{Q}_\infty[X^{{\mathbb{Z}_\infty^+}}]\setminus\mathcal{I}=\{[\mathcal{I}]_\mathfrak{p}:\mathfrak{p}\in\mathbb{Q}_\infty[X^{{\mathbb{Z}_\infty^+}}]\}.$$  Further, since $[\mathcal{I}]$ is a congruence relation, we may canonically define addition \\ $\hat+:(\mathbb{Q}_\infty[X^{{\mathbb{Z}_\infty^+}}]\setminus\mathcal{I})\times(\mathbb{Q}_\infty[X^{{\mathbb{Z}_\infty^+}}]\setminus\mathcal{I})\rightarrow\mathbb{Q}_\infty[X^{{\mathbb{Z}_\infty^+}}]\setminus\mathcal{I}$ by  $$[\mathcal{I}]_\mathfrak{p}\hat+[\mathcal{I}]_\mathfrak{q}=[\mathcal{I}]_{\mathfrak{p}\hat+\mathfrak{q}},$$ where we have abused the symbol $\hat+$ since it canonically induces the additive structure on $\mathbb{Q}_\infty[X^{{\mathbb{Z}_\infty^+}}]\setminus\mathcal{I}$. Similarly, we may define negation $\hat-:\mathbb{Q}_\infty[X^{{\mathbb{Z}_\infty^+}}]\setminus\mathcal{I}\rightarrow\mathbb{Q}_\infty[X^{{\mathbb{Z}_\infty^+}}]\setminus\mathcal{I}$ by $$\hat-[\mathcal{I}]_\mathfrak{p}=[\mathcal{I}]_{\hat-\mathfrak{p}},$$ and we define multiplication $\hat\times:(\mathbb{Q}_\infty[X^{{\mathbb{Z}_\infty^+}}]\setminus\mathcal{I})\times(\mathbb{Q}_\infty[X^{{\mathbb{Z}_\infty^+}}]\setminus\mathcal{I})\rightarrow\mathbb{Q}_\infty[X^{{\mathbb{Z}_\infty^+}}]\setminus\mathcal{I}$ by $$[\mathcal{I}]_\mathfrak{p}\hat\times[\mathcal{I}]_\mathfrak{q}=[\mathcal{I}]_{\mathfrak{p}\hat\times\mathfrak{q}}.$$ In light of these nice factorings for our algebraic operations, we may then define the ordering $\preceq$ on $\mathbb{Q}_\infty[X^{{\mathbb{Z}_\infty^+}}]\setminus\mathcal{I}$ as $$[\mathcal{I}]_\mathfrak{p}\prec[\mathcal{I}]_\mathfrak{q}\iff p\prec q\wedge[\mathcal{I}_\mathfrak{p}\neq\mathcal{I}_\mathfrak{q}].$$ The class $\mathbb{Q}_\infty[X^{{\mathbb{Z}_\infty^+}}]\setminus\mathcal{I}$ together with the above algebraic and ordering structure will be called the {\bf quotient ring of $\mathbb{Q}_\infty[X^{{\mathbb{Z}_\infty^+}}]$ generated by $\mathcal{I}$}. \\
\end{defn}

That the above construction actually gives rise to an ordered ring/field is a very standard proof in Galois theory once we have established that $\mathbb{Q}_\infty[X^{{\mathbb{Z}_\infty^+}}]$ is an ordered Ring, as we did in Theorem {\it 4.1}. The only nuanced pieces of the proof machinery that must be lifted to this setting are the requirements that the ideals generated by irreducible naked polynomials still be maximal, so that their quotients will be fields, and that $\mathbb{Q}_\infty[X^{\mathbb{Z}_\infty^+}]$ be an Euclidean domain. \\

\begin{thm}
Let $\mathfrak{p}$ be an irreducible naked polynomial; then $\mathcal{I}_\mathfrak{p}$ is a maximal ideal in $\mathbb{Q}_\infty[X^{{\mathbb{Z}_\infty^+}}]$.
\end{thm}
\begin{proof}
First, we show that $\mathcal{I}_\mathfrak{p}$ is an ideal. Suppose that $\mathfrak{q},\mathfrak{r}\in\mathcal{I}_\mathfrak{p}$ with $\mathfrak{q}\neq\mathfrak{p}\neq\mathfrak{r}$, so there exist $\mathfrak{q}'\neq\hat1\neq\mathfrak{r}'$ in $\mathbb{Q}_\infty[X^{{\mathbb{Z}_\infty^+}}]$ such that $\mathfrak{p}\hat\times\mathfrak{q}'=\mathfrak{q}$ and $\mathfrak{p}\hat\times\mathfrak{r}'=\mathfrak{r}$. We wish to show that $\mathfrak{q}\hat-\mathfrak{r}\in\mathcal{I}_\mathfrak{p}$, and indeed $$\mathfrak{q}\hat-\mathfrak{r}=\mathfrak{p}\hat\times\mathfrak{q}'\hat-\mathfrak{p}\hat\times\mathfrak{r}'=\mathfrak{p}\hat\times(\mathfrak{q}'\hat-\mathfrak{r}')\implies\mathfrak{p}\in\{\mathfrak{p},\mathfrak{q}'\hat-\mathfrak{r}'\}\in\mathcal{F}_{\mathfrak{q}\hat-\mathfrak{r}}\implies\mathfrak{q}\hat-\mathfrak{r}\in\mathcal{I}_\mathfrak{p}.$$ Now, suppose that $\mathfrak{s}\in\mathbb{Q}_\infty[X^{{\mathbb{Z}_\infty^+}}]$. We then have that $$\mathfrak{s}\hat\times\mathfrak{q}=\mathfrak{s}\hat\times(\mathfrak{p}\hat\times\mathfrak{q}')=\mathfrak{p}\hat\times(\mathfrak{s}\hat\times\mathfrak{q}'),$$ thus $\mathfrak{s}\hat\times\mathfrak{q}\in\mathcal{I}_\mathfrak{p}$. This completes the proof that $\mathcal{I}_\mathfrak{p}$ is closed under differences and multiplication by elements of $\mathbb{Q}_\infty[X^{{\mathbb{Z}_\infty^+}}]$, thus $\mathcal{I}_\mathfrak{p}$ is an ideal in $\mathbb{Q}_\infty[X^{{\mathbb{Z}_\infty^+}}]$. To see that $\mathcal{I}_\mathfrak{p}$ is maximal, we simply observe that $\mathcal{F}_\mathfrak{p}=\{\{\mathfrak{p},\hat1\}\}$ since $\mathfrak{p}$ is irreducible, and $\mathcal{I}_{\hat1}=\mathbb{Q}_\infty[X^{{\mathbb{Z}_\infty^+}}]$ since $\hat1\in\{\hat1,\mathfrak{q}\}\in\mathcal{F}_\mathfrak{q}$ for all $\mathfrak{q}\in\mathbb{Q}_\infty[X^{{\mathbb{Z}_\infty^+}}]$, so there are no ideals strictly between $\mathcal{I}_\mathfrak{p}$ and $\mathbb{Q}_\infty[X^{{\mathbb{Z}_\infty^+}}]$. This completes the proof. \\
\end{proof}

As usual, $\big(x\big)=\big(({_\beta}[1]_1)_{\beta\in\mathbb{Z}_\infty^+}\big)$ is the largest ideal. In order to carry out the proper recursion for the next theorem showing that $\mathbb{Q}_\infty[X^{\mathbb{Z}_\infty^+}]$ is an Euclidean domain, we must first define a generalization of $_\mathfrak{p}\uparrow_\mathfrak{q}$ which is amenable to recursion.  \\

\begin{defn}
For all $\mathfrak{p}\in\mathbb{Q}_\infty[X^{\mathbb{Z}_\infty^+}]$ we define $\overline{|\mathfrak{p}|}\in O_n$ to be the unique ordinal which is in bijection with $supp(\mathfrak{p})$, guaranteed by the counting principle which is an equivalent of the standard axiom of choice ([$2$]). Note that $supp(\mathfrak{p})$ must be a set, since it has a maximal element $_{\hat0}\uparrow_\mathfrak{p}$ and there are only some set-sized ordinal number of positions between any fixed $n\in\mathbb{Z}_\infty^+$ and $0$, since it is discretely ordered -- consider the Surinteger normal form of $n$ to determine the number of positions. For all $\mathfrak{p},\mathfrak{q}\in\mathbb{Q}_\infty[X^{\mathbb{Z}_\infty^+}]$, we then define a sequence $\{{_\mathfrak{p}\uparrow^\alpha_\mathfrak{q}}\}_{\alpha<\zeta}$ by recursion on $\zeta=\max\{\overline{|\mathfrak{p}|},\overline{|\mathfrak{q}|}\}$ as follows: $$_\mathfrak{p}\uparrow^0_\mathfrak{q}={_\mathfrak{p}\uparrow_\mathfrak{q}},$$ $$_\mathfrak{q}\uparrow^{\alpha}_\mathfrak{p}=\max\Big(\{n\in\mathbb{Z}_\infty^+:\mathfrak{p}_n\neq\mathfrak{q}_n\}\setminus\{_\mathfrak{p}\uparrow^m_\mathfrak{q}:m<\alpha\}\Big).$$ Note that $\big(\{n\in\mathbb{Z}_\infty^+:\mathfrak{p}_n\neq\mathfrak{q}_n\}\setminus\{_\mathfrak{p}\uparrow^m_\mathfrak{q}:m<\alpha\}\big)\subseteq\big(supp(\mathfrak{p})\cup supp(\mathfrak{q})\big)$, so $_\mathfrak{p}\uparrow^\alpha_\mathfrak{q}$ is well defined for all $\alpha<\zeta$. This recursion lists the differing coordinates of $\mathfrak{p}$ and $\mathfrak{q}$ in decreasing order. Setting $\mathfrak{p}=\hat0$, we obtain a recursion that lists the non-zero coordinates of $\mathfrak{q}$ in decreasing order -- we will omit the $\hat0$ in this case and simply write $\{\uparrow^\alpha_\mathfrak{q}\}_{\alpha<\overline{|\mathfrak{q}|}}$, and we will call this sequence the {\bf coordinate sequence of $\mathfrak{q}$}. \\
\end{defn}

\begin{thm}
For all $\mathfrak{p}\in\mathbb{Q}_\infty[X^{\mathbb{Z}_\infty^+}]$, we now have a recursively defined representation for $\mathfrak{p}$: $$\mathfrak{p}=\Big(\sum_{\alpha<\overline{|\mathfrak{p}|}}{_\beta[\mathfrak{p}_{\uparrow^\alpha_\mathfrak{p}}]_{\uparrow^\alpha_\mathfrak{p}}}\Big)_{\beta\in\mathbb{Z}_\infty^+}.$$
\end{thm}
\begin{proof}
By the definitions involved, this places $\mathfrak{p}_{\uparrow^\alpha_\mathfrak{p}}$ in position $\uparrow^\alpha_\mathfrak{p}$ for all $\alpha<\overline{|\mathfrak{p}|}$, and these are precisely all of the non-zero coordinates of $\mathfrak{p}$. \\
\end{proof}

\begin{defn}
For all $\mathfrak{p}\in\mathbb{Q}_\infty[X^{\mathbb{Z}_\infty^+}]$, we will refer to the representation $$\mathfrak{p}=\Big(\sum_{\alpha<\overline{|\mathfrak{p}|}}{_\beta[\mathfrak{p}_{\uparrow^\alpha_\mathfrak{p}}]_{\uparrow^\alpha_\mathfrak{p}}}\Big)_{\beta\in\mathbb{Z}_\infty^+}$$ given in the previous theorem as the {\bf recursive representation} of $\mathfrak{p}$. Further, when expressing $\mathfrak{p}$ as its recursive representation we will omit the extra subscript on $\uparrow^\alpha_\mathfrak{p}$, yielding $$\mathfrak{p}=\Big(\sum_{\alpha<\overline{|\mathfrak{p}|}}{_\beta[\mathfrak{p}_{_{\uparrow^\alpha}}]_{\uparrow^\alpha}}\Big)_{\beta\in\mathbb{Z}_\infty^+}.$$\\
\end{defn}

The recursive representation of a naked polynomial is necessary for the computations in the next theorem. \\

\begin{thm}
For all $\mathfrak{p},\mathfrak{q}\in\mathbb{Q}_\infty[X^{\mathbb{Z}_\infty^+}]$ such that $\uparrow^0_\mathfrak{q}\leq\uparrow^0_\mathfrak{p}$, there exist unique $\mathfrak{r},\mathfrak{s}\in\mathbb{Q}_\infty[X^{\mathbb{Z}_\infty^+}]$ such that $$\mathfrak{p}=\mathfrak{q}\hat\times\mathfrak{r}\hat+\mathfrak{s},$$ and $\mathfrak{s}=\hat0$ or $\uparrow^0_\mathfrak{s}<\uparrow^0_\mathfrak{q}$. Accordingly, $\mathbb{Q}_\infty[X^{\mathbb{Z}_\infty^+}]$ is an Euclidean domain with $norm:\mathbb{Q}_\infty\rightarrow\mathbb{Z}_\infty^+$ given by $$norm(\mathfrak{p})=\uparrow^0_\mathfrak{p}$$ for all $\mathfrak{p}\in\mathbb{Q}_\infty[X^{\mathbb{Z}_\infty^+}]$.
\end{thm}
\begin{proof}
We will show that $\mathbb{Q}_\infty[X^{\mathbb{Z}_\infty^+}]$ is closed under the division recursion defined below (this recusrion is called an {\it algorithm} when it is finite). For all $\mathfrak{p},\mathfrak{q}\in\mathbb{Q}_\infty[X^{\mathbb{Z}_\infty^+}]$ such that $\uparrow_\mathfrak{q}\leq{\uparrow_\mathfrak{p}}$, we define $\mathfrak{r}=\mathfrak{q}\lceil\mathfrak{p}$ and $\mathfrak{s}$ by recursion as follows: $$\mathfrak{q}\lceil\mathfrak{p}^0=\big({_\beta[\frac{\mathfrak{p}_{_{\uparrow^0}}}{\mathfrak{q}_{_{\uparrow^0}}}]_{\uparrow^0_\mathfrak{p}-\uparrow^0_\mathfrak{q}}}\big)_{\beta\in\mathbb{Z}_\infty^+},$$ $$\mathfrak{s}^0=\mathfrak{p}\hat-\mathfrak{q}\hat\times\big({_\beta[\frac{\mathfrak{p}_{_{\uparrow^0}}}{\mathfrak{q}_{_{\uparrow^0}}}]_{\uparrow^0_\mathfrak{p}-\uparrow^0_\mathfrak{q}}}\big)_{\beta\in\mathbb{Z}_\infty^+},$$ $$\mathfrak{q}\lceil\mathfrak{p}^\alpha=\Big(\sum_{i<\alpha}{_\beta}[\mathfrak{q}\lceil\mathfrak{p}^i]_{\uparrow^i}\Big)_{\beta\in\mathbb{Z}_\infty^+}+\big({_\beta}[\frac{\mathfrak{s}^{\alpha-1}_{_{\uparrow^0}}}{\mathfrak{q}_{_{\uparrow^0}}}]_{\uparrow^0_{\mathfrak{s}^{\alpha-1}}-\uparrow^0_\mathfrak{q}}\big)_{\beta\in\mathbb{Z}_\infty^+},\ \text{if}\ \alpha\ \text{is a successor ordinal},$$ $$\mathfrak{s}^\alpha=\mathfrak{s}^{\alpha-1}\hat-\mathfrak{q}\hat\times\big({_\beta[\frac{\mathfrak{s}^{\alpha-1}_{_{\uparrow^0}}}{\mathfrak{q}_{_{\uparrow^0}}}]_{\uparrow^0_{\mathfrak{s}^{\alpha-1}}-\uparrow^0_\mathfrak{q}}}\big)_{\beta\in\mathbb{Z}_\infty^+},\ \text{if}\ \alpha\ \text{is a successor ordinal},$$ $$\mathfrak{q}\lceil\mathfrak{p}^{\gamma}=\Big(\sum_{\alpha<\gamma}{_\beta}[\mathfrak{q}\lceil\mathfrak{p}^\alpha]_{\uparrow^\alpha}\Big)_{\beta\in\mathbb{Z}_\infty^+},\ 0<\gamma\ \text{is a limit ordinal},$$ $$\mathfrak{s}^{\gamma}=\mathfrak{p}\hat-\mathfrak{q}\hat\times\mathfrak{q}\lceil\mathfrak{p}^{\gamma},\ 0<\gamma\ \text{is a limit ordinal}.$$ Note that the recursion taking place here only terminates at a successor step if $\uparrow^0_{\mathfrak{s}^\alpha}<\uparrow^0_\mathfrak{q}$ for some successor ordinal $\alpha$ -- in other words it only terminates if the degree of $\mathfrak{s}^\alpha$ is less than the degree of $\mathfrak{q}$ for some successor $\alpha$, which only happens if $\uparrow^\beta_\mathfrak{p}$ and $\uparrow^0_\mathfrak{q}$ are comparable for some $\beta$ such that $\uparrow^0_\mathfrak{q}\leq\uparrow^\beta_\mathfrak{p}$. If $\uparrow^0_\mathfrak{q}<<\uparrow^\beta_\mathfrak{p}$ for all $\beta$, or if $\uparrow^0_\mathfrak{q}<<\uparrow^\alpha_\mathfrak{p}$ for all $\alpha<\beta$ and $\uparrow^0_\mathfrak{q}>\uparrow^\beta_\mathfrak{p}$ for some $0<\beta\in O_n$, the recursion terminates at the limit ordinal step $\omega\cdot\beta$. Further, since all supports are sets by the observations in Definition {\it 4.5}, the recursion must terminate at some ordinal step. Denote by $\overline{|\mathfrak{q}\lceil\mathfrak{p}|}$ the unique ordinal step at which this recursion terminates, noting that $0<\overline{|\mathfrak{q}\lceil\mathfrak{p}|}$ since $\uparrow^0_\mathfrak{q}\leq\uparrow^0_\mathfrak{p}$; we then define $\mathfrak{q}\lceil\mathfrak{p}^{\overline{|\mathfrak{q}\lceil\mathfrak{p}|}}=\mathfrak{q}\lceil\mathfrak{p}=\mathfrak{r}$, and $\mathfrak{s}^{\overline{|\mathfrak{q}\lceil\mathfrak{p}|}}=\mathfrak{s}$. By the above observations on when this recursion terminates, we have that $\uparrow_\mathfrak{s}<{\uparrow}_\mathfrak{q}$ if $\overline{|\mathfrak{q}\lceil\mathfrak{p}|}$ is a successor ordinal. Further we note that $\mathfrak{s}=\hat0\iff\mathfrak{q}\hat\times(\mathfrak{q}\lceil\mathfrak{p})=\mathfrak{p}$, which happens precisely iff $\{\mathfrak{q},\mathfrak{q}\lceil\mathfrak{p}\}\in\mathcal{F}_\mathfrak{p}$, which implies that $$\mathfrak{q}\hat\times\mathfrak{r}\hat+\mathfrak{s}=\Big(\mathfrak{q}\hat\times(\mathfrak{q}\lceil\mathfrak{p})\Big)\hat+\Big(\mathfrak{p}\hat-\mathfrak{q}\hat\times(\mathfrak{q}\lceil\mathfrak{p})\Big)=\mathfrak{p}\hat+0=\mathfrak{p}.$$ Suppose that $\{\mathfrak{q},\mathfrak{q}\lceil\mathfrak{p}\}\notin\mathcal{F}_\mathfrak{p}$, so $\mathfrak{s}\neq\hat0$. Recall now that $\zeta\in O_n$ is a limit ordinal iff $\zeta=\omega\cdot\gamma$ for some $\gamma\in O_n$, where $\cdot$ denotes recursive ordinal multiplication ([$2$]). Suppose the recursion terminates at a limit ordinal, so $\overline{|\mathfrak{q}\lceil\mathfrak{p}|}=\omega\cdot\gamma$ for some $0<\gamma\in O_n$. We then have that $\mathfrak{q}\hat\times\mathfrak{q}\lceil\mathfrak{p}$ is exactly the highest $\gamma$ nonzero support terms of $\mathfrak{p}$, which were incomparable $\uparrow^0_\mathfrak{q}$. More precisely, $\mathfrak{q}\hat\times\mathfrak{q}\lceil\mathfrak{p}=(\sum_{\alpha<\gamma}{_\beta}[\mathfrak{p}_{\uparrow^\alpha}]_{\uparrow^\alpha_\mathfrak{p}})_{\beta\in\mathbb{Z}_\infty^+}$, thus we have that $$\mathfrak{s}=\mathfrak{p}\hat-\mathfrak{q}\hat\times\mathfrak{q}\lceil\mathfrak{p}=(\sum_{\alpha<\overline{|\mathfrak{p}|}}{_\beta}[\mathfrak{p}_{\uparrow^\alpha}]_{\uparrow^\alpha_\mathfrak{p}})_{\beta\in\mathbb{Z}_\infty^+}\hat-(\sum_{\alpha<\gamma}{_\beta}[\mathfrak{p}_{\uparrow^\alpha}]_{\uparrow^\alpha_\mathfrak{p}})_{\beta\in\mathbb{Z}_\infty^+}=(\sum_{\gamma\leq\alpha<\overline{|\mathfrak{p}|}}{_\beta}[\mathfrak{p}_{\uparrow^\alpha}]_{\uparrow^\alpha_\mathfrak{p}})_{\beta\in\mathbb{Z}_\infty^+}.$$ Since $\mathfrak{s}\neq\hat0$ we have that $\gamma<\overline{|\mathfrak{p}|}$, and further we then have that $\uparrow^0_\mathfrak{s}=\uparrow^\gamma_\mathfrak{p}<\uparrow^0_\mathfrak{q}$ since the recursion terminated. We also have that $$\mathfrak{q}\hat\times\mathfrak{q}\lceil\mathfrak{p}\hat+\mathfrak{s}=(\sum_{\alpha<\gamma}{_\beta}[\mathfrak{p}_{\uparrow^\alpha}]_{\uparrow^\alpha_\mathfrak{p}})_{\beta\in\mathbb{Z}_\infty^+}\hat+(\sum_{\gamma\leq\alpha<\overline{|\mathfrak{p}|}}{_\beta}[\mathfrak{p}_{\uparrow^\alpha}]_{\uparrow^\alpha_\mathfrak{p}})_{\beta\in\mathbb{Z}_\infty^+}=(\sum_{\alpha<\overline{|\mathfrak{p}|}}{_\beta}[\mathfrak{p}_{\uparrow^\alpha}]_{\uparrow^\alpha_\mathfrak{p}})_{\beta\in\mathbb{Z}_\infty^+}=\mathfrak{p}.$$ Now suppose the recursion terminates at a successor ordinal step, so $\uparrow^0_\mathfrak{s}<\uparrow^0_\mathfrak{q}$ by our previous observation; we now wish to show that the desired equality still holds. But we may now observe that, since the last limit ordinal step where we already know it was well behaved, this has simply been the standard division algorithm of a polynomial ring. Conesquently $\mathfrak{p}=\mathfrak{q}\hat\times\mathfrak{q}\lceil\mathfrak{p}\hat+\mathfrak{s}=\mathfrak{q}\hat\times\mathfrak{r}\hat+\mathfrak{s}$ once again, and since $\mathfrak{p}$ and $\mathfrak{q}$ were arbitrary and $\mathfrak{r}$ and $\mathfrak{s}$ are clearly unique by their recursive parametrization from $\mathfrak{p}$ and $\mathfrak{q}$, this completes the proof. \\
\end{proof}

\begin{note}
We now give a few examples which the reader can check for themselves to hopefully gain some intuition. First, note that the above recursion is simply a transfinite version of the standard 'division algorithm' endowed on a polynomial ring. All of the following recursions can be intuitively carried out by writing $\mathfrak{p}$ inside the standard polynomial long-division symbol and $\mathfrak{q}$ outside the symbol, where $\mathfrak{q}\lceil\mathfrak{p}$ will then appear by recursion on top of the symbol using the standard long-division algorithm extended to transfinitely many steps -- $\mathfrak{s}$ is the remainder left over.  For example,
\begin{itemize}
\item Suppose $\mathfrak{p}=({_\beta}[1]_\omega)_{\beta\in\mathbb{Z}_\infty^+}=X^\omega$, and $\mathfrak{q}=({_\beta}[1]_2+{_\beta}[1]_0)_{\beta\in\mathbb{Z}_\infty^+}=X^2+1$. We then have that $$\mathfrak{q}\lceil\mathfrak{p}^0=({_\beta}[1]_{\omega-2})_{\beta\in\mathbb{Z}_\infty^+}=X^{\omega-2},$$ $$\mathfrak{s}^0=\mathfrak{p}\hat-\mathfrak{q}\hat\times\mathfrak{q}\lceil\mathfrak{p}^0=X^\omega\hat-\Big((X^2+1)\hat\times X^{\omega-2}\Big)=X^\omega\hat-(X^\omega+X^{\omega-2})=-X^{\omega-2},$$ $$\mathfrak{q}\lceil\mathfrak{p}^1=({_\beta}[1]_{\omega-2}+{_\beta}[-1]_{\omega-4})_{\beta\in\mathbb{Z}_\infty^+}=X^{\omega-2}-X^{\omega-4},$$ $$\mathfrak{s}^1=\mathfrak{s}^0\hat-(\mathfrak{q}\hat\times X^{\omega-4})=-X^{\omega-2}\hat-\Big((X^2+1)\hat\times- X^{\omega-4}\big)=-X^{\omega-2}\hat-(-X^{\omega-2}-X^{\omega-4})=X^{\omega-4},$$ $$\dots$$ $$\mathfrak{q}\lceil\mathfrak{p}^n=\Big(\sum_{i\leq n}{_\beta}[(-1)^i]_{\omega-2(i+1)}\Big)_{\beta\in\mathbb{Z}_\infty^+}=\sum_{i\leq n}(-1)^iX^{\omega-2(i+1)},$$ $$\mathfrak{s}^n=(-1)^nX^{\omega-2(n+1)},$$ $$\dots$$ $$\mathfrak{q}\lceil\mathfrak{p}^\omega=\mathfrak{q}\lceil\mathfrak{p}=\Big(\sum_{n<\omega}{_\beta}[(-1)^n]_{\omega-2(n+1)}\Big)_{\beta\in\mathbb{Z}_\infty^+}=\sum_{n<\omega}(-1)^nX^{\omega-2(n+1)},$$ $$\mathfrak{s}^\omega=\mathfrak{s}=\mathfrak{p}-\mathfrak{q}\hat\times\mathfrak{q}\lceil\mathfrak{p}.$$ $\mathfrak{q}\hat\times\mathfrak{q}\lceil\mathfrak{p}$ is now a telescoping sum in which all terms but the leading term, $X^\omega$, cancel out: $$\mathfrak{q}\hat\times\mathfrak{q}\lceil\mathfrak{p}=(X^2+1)\hat\times\Big(\sum_{n<\omega}(-1)^nX^{\omega-2(n+1)}\Big)=\sum_{n<\omega}(-1)^n(X^{\omega-2n}+X^{\omega-2(n+1)})$$ $$=(X^\omega+X^{\omega-2})-(X^{\omega-2}+X^{\omega-4})+(X^{\omega-4}+X^{\omega-6})-\dots=X^\omega=\mathfrak{p}.$$ We also have that $\mathfrak{s}=\hat0$, since $X^2+1$ divides $X^\omega$ exactly, and in general if $\{\uparrow^0_\mathfrak{q}\}<<supp(\mathfrak{p})$  then $\mathfrak{q}$ will divide $\mathfrak{p}$ exactly using these telescoping sums. We carry out the calculations dressed up because the intuition is easier, but keep in mind that these are really functions from $\mathbb{Q}_\infty$ into $\mathbb{Z}_\infty$ such that all subsets of the support have a maximal element. \\
\item Now, let $\mathfrak{p}=X^{\omega^\omega}$ and $\mathfrak{q}=\sum_{n<\omega}X^{\omega-3(n+1)}$. The reader can easily check that we obtain \\ $\mathfrak{q}\lceil\mathfrak{p}=\sum_{n<\omega}X^{\omega^\omega-\omega+3(1-n)}$, and further that $$\Big(\sum_{n<\omega}X^{\omega-3(n+1)}\Big)\hat\times\Big(\sum_{n<\omega}X^{\omega^\omega-\omega+3(1-n)}\Big)=X^{\omega^\omega},$$ so we see that the infinitely decreasing sum $\sum_{n<\omega}X^{\omega-3(n+1)}$ divides $X^{\omega^\omega}$ exactly, as was asserted above since $\{\omega-3(n+1)\}<<\{\omega^\omega\}$ for all $n<\omega$. For a nonzero remainder, we could simply make $\mathfrak{p}=X^{\omega^\omega}+X^n$ for any finite $n$, whereupon we will obtain $\mathfrak{q}\lceil\mathfrak{p}$ the same as before with $\mathfrak{s}=X^n$.\\
\end{itemize}
\end{note}

Many more rich and instructive examples can be found by making more interesting/complicated selections for $\mathfrak{p}$ and $\mathfrak{q}$, and exploration of these examples is encouraged -- this is initially how I came across the recursive structure defined in this paper. All that remains from a theoretical existence standpoint is to show that the non-constructive process of algebraic closure can be extended to the Surrational numbers in a canonical fashion -- this requires a version of Zorn's lemma which may be safely applied to proper classes. This is equivalent to the axiom of Global choice (GC). \\

\begin{note}
We now add the axiom of Global choice to whatever ambient set theory we are working in. \\
\end{note}

\begin{thm}
All fields admit algebraic closures. In particlar, $\mathbb{Q}_\infty$ has an algebraic closure, denoted $\overline{\mathbb{Q}}_\infty$.
\end{thm}
\begin{proof}
By the proper class version of Zorn's lemma we have that the class of all field extensions of $\mathbb{Q}_\infty$, denoted $\mathbb{Q}_\infty[\infty]$ and ordered by inclusion embeddings, has at least one maximal element since each totally ordered subset of $\mathbb{Q}_\infty[\infty]$ has a maximal element in $\mathbb{Q}_\infty[\infty]$ generated by an irreducible polynomial, since $\mathbb{Q}_\infty[X^{\mathbb{Z}_\infty^+}]$ is an Euclidean domain and consequently a principal ideal domain. This comlpetes the proof.  \\
\end{proof}

\begin{defn}
We will refer to $\overline{\mathbb{Q}}_\infty$ as the {\bf algebraic Surrational numbers}. \\
\end{defn}

Here we conclude my exposition on this piece of the theory.  I intend to publish subsequent papers on the structure of various Galois groups over $\mathbb{Q}_\infty$ up to and including its absolute Galois group at a later date, but for now I intend to move on and construct the Surreals out of the algebraic Surrational numbers. I would like to profusely thank professors Andrew Conner and Charles Hamaker at St. Mary's College of California -- without their steadfast support and advice, this project likely would not have come to fruition so quickly. \\

\section{References}
\begin{enumerate}
\item {\it The Ordinals as a Consummate Abstraction of Number Systems, Alec Rhea}, 2017,  \\ https://arxiv.org/abs/1706.08908.
\item {\it Introduction to Set Theory, J. Donald Monk}, 1969.
\item {\it Fields and Galois Theory, J.S. Milne}, 2017, \\ http://www.jmilne.org/math/CourseNotes/FT.pdf.
\end{enumerate}

\end{document}